%% file: main.tex
\input{preamble.tex}

\begin{document}

\title{
    \huge{A framework for spanning forest identities}
}
\author{Kai Choi, Rowan Yeats\thanks{KC was supported by an NSERC USRA. RY thanks Melanie Fraser for many useful discussions.  RY is partially supported by an NSERC Discovery grant and the Canada Research Chairs program.}}

\maketitle

\begin{abstract}
Identities of spanning forest polynomials can come from determinantal identities such as the classical Dodgson identity. They can also be useful in quantum field theory calculations, particularly the quadratic ones.  To better understand the structure of spanning forest polynomial identities, we look at when identities on a given graph are generic for all graphs, we ask about dimensions of spaces of spanning forest identities indexed by certain parameters, and we set up a framework that we hope will be helpful for further progress in these directions. As part of this framework, we identify a family of maps naturally given by the all-minors matrix-tree theorem, and we show that their ranks are counted by the Narayana numbers via explicit construction of bases.
\end{abstract}

\input{intro.tex}

\input{sf-identities.tex}

\input{mtt-det-identities.tex}

\input{narayana.tex}

\appendix
\input{muir-pf.tex}
\bibliographystyle{plain}
\bibliography{main}

\end{document}

%% file: preamble.tex
\documentclass[11et, letterpaper]{article}
\usepackage[left=2.5cm,right=2.5cm, top=2.5cm,bottom=2.5cm,bindingoffset=0cm]{geometry}
\usepackage{indentfirst}
\usepackage{graphicx} 
\usepackage{amsmath}
\usepackage{amssymb}
\usepackage{amsthm}
\usepackage{xcolor}
\usepackage{enumitem}
\usepackage{mathtools}
\usepackage{nameref}
\usepackage{titlesec}
\usepackage{bbold}
\usepackage{tikz-cd}
\usepackage{caption}
\usepackage{subcaption}
\usepackage[utf8]{inputenc}
\usepackage{bm}

\usepackage{hyperref}
\hypersetup{
    colorlinks=true,
    linkcolor=blue,
    filecolor=magenta,      
    urlcolor=cyan,
    pdftitle={Overleaf Example},
    pdfpagemode=FullScreen,
    }
\newcommand{\mb}[1]{\mathbb{#1}}
\newcommand{\mc}[1]{\mathcal{#1}}

\newcommand{\eps}{\varepsilon}

\DeclarePairedDelimiter{\brdelim}{\{}{\}}
\DeclarePairedDelimiter{\prdelim}{(}{)}

\DeclarePairedDelimiter{\abdelim}{|}{|}

\newcommand{\br}{\brdelim*}
\newcommand{\pr}{\prdelim*}

\newcommand{\ab}{\abdelim*}

\DeclareMathOperator{\im}{im}
\DeclareMathOperator{\rank}{rank}
\DeclareMathOperator{\spann}{span}

\newtheorem{theorem}{Theorem}[section]
\newtheorem{lemma}[theorem]{Lemma}
\newtheorem{proposition}[theorem]{Proposition}
\newtheorem{corollary}[theorem]{Corollary}

\theoremstyle{definition}
\newtheorem{definition}[theorem]{Definition}
\newtheorem{remark}[theorem]{Remark}
\newtheorem{question}[theorem]{Question}

\newtheorem{examplex}[theorem]{Example}
\newenvironment{example}
  {\pushQED{\qed}\examplex}
  {\popQED\endexamplex}

  \numberwithin{equation}{section}

%% file: intro.tex
\section{Introduction}

The classical Dodgson identity is a quadratic determinantal identity popularized by Dodgson in his condensation algorithm \cite{Dodgson}. Specifically, we have
\[
 \det M \det M(\{1,2\}|\{1,2\}) = \det M(\{1\}| \{1\}) \det M(\{2\}| \{2\}) - \det M(\{1\}| \{2\}) \det M(\{2\}| \{1\}),
 \]
 where $M$ is a square matrix with at least two rows and columns, and where $M(S|T)$ is the matrix $M$ with rows indexed by $S$ and columns indexed by $T$ removed.
 If $M$ is the Laplacian matrix of a graph or digraph (or a similar matrix defined from a graph), then this identity can be reinterpreted as a quadratic identity of spanning forests of the graph, as has been studied from a few different perspectives; see \cite{by, ch, fraser}.
  
One of the contexts in which this identity is important is in parametric Feynman integration in quantum field theory \cite{br}.  When calculating certain integrals known as Feynman periods of scalar $\phi^4$ Feynman graphs, the Dodgson identity allows us to rewrite the denominator at the second integration in a way that permits the next integration to proceed. Likewise, to perform subsequent integration steps, it is useful to consider other quadratic identities of minors of the expanded Laplacian matrix \cite{Ybook}, a variant on the Laplacian matrix of a graph. As one of us showed with Francis Brown \cite{by}, these minors can be written as signed sums of \emph{spanning forest polynomials} (defined in Section~\ref{sec:sf-identities}), yielding quadratic spanning forest polynomial identities.  

In view of this, we are interested in better understanding the space of quadratic spanning forest polynomial identities.  Aside from the quantum field theory application, we think this is a natural combinatorial question to investigate.  Important questions that are not in general understood include: what is the dimension of the space of spanning forest polynomial identities of a given degree for a given graph?  Do all identities of spanning forest polynomials come from determinantal identities?  Our spanning forest polynomials are defined in terms of a set partition of a subset of the vertices of the graph, so an important parameter will be the number of \emph{marked vertices} from which the set partition must be taken.  Restricting for the moment to complete graphs (other graphs can be obtained by setting some edge variables to zero), as the number of vertices of the graph increases but the number of marked vertices remains fixed, at what point is the graph large enough to be generic in the sense that all quadratic spanning forest polynomial identities of that graph are identities of all graphs with the same number of marked vertices?

Previous work of one of us with Melanie Fraser \cite{fy} showed that the space of quadratic spanning forest polynomials involving $4$ marked vertices, each of whose terms is a product of a spanning forest polynomial whose set partition has 2 parts and one whose set partition has $m-1$ parts, has dimension $m(m-2)$.  This generalized results of and proved a conjecture of the same one of us with Aleksandar Vlasev \cite{vy}.  The fact that the total number of parts between both partitions sums to $m+1$ is inspired by the quantum field theory context; even with that restriction, all other cases remain unknown.  Numbers of parts with other sums are also mathematically interesting.

These questions contrast somewhat with the usual generalizations of the Dodgson identity in the literature (see \cite{bs}, \cite{dennis}, and references therein) where the degree of the identity go up as the number of removed rows and columns (or equivalently, the number of marked vertices) goes up (one example which does not is \cite{Williamson}, which is closely related to the above mentioned quadratic results of \cite{fy}).  The growth of the degree with the number of marked vertices gives the more usual generalizations a somewhat different flavour and also leads to asking about generalizations of other degrees that do not grow with the number of marked vertices.  Other degrees are also of interest to physics when we look at Feynman integrals outside the limited context of Feynman periods in scalar $\phi^4$ theory, but for a given problem, the degree remains fixed while the number of marked vertices grows at each integration step.

In this paper, we show that the complete graph on the marked vertices is not generally large enough to be generic for quadratic spanning forest polynomial identities using an explicit counterexample, contrary to expectations from previous work. Also, we give an algebraic set-up that is useful for reasoning about spanning forest identities. In particular, we define a family of \emph{MTT maps} that interpret graph Laplacian minor identities as spanning forests identities; such maps are useful to consider when investigating whether all spanning forest identities can be deduced from determinantal identities of Laplacian matrices. We then give explicit bases for the degree-1 MTT maps and show that the ranks of these maps are counted by the Narayana numbers.

\subsection{Set-up and first few Feynman period integrations}

Throughout, we will fix a field $F$ of characteristic zero and work with simple undirected graphs $G$ with vertex set $[n] = \br{1, 2, \ldots, n}$.  We often fix some $m \le n$ and call the vertices $[m] \subseteq [n]$ of $G$ \emph{marked vertices}.  The marked vertices will be used to index the spanning forest polynomials and the minors of the graph Laplacian, which are the key objects of study.

Given a graph $G$ on $[n]$, its {\em Laplacian matrix} is the $n$-by-$n$ matrix $L_G$ with entries in $F[E(G)]$ given by
\begin{align*}
    (L_G)_{ij} &= \begin{cases}
        x_{ij} & \text{if } i \neq j \text{ and } ij \in E(G) \\
        0 & \text{if } i \neq j \text{ and } ij \not\in E(G) \\
        -\sum_{ik \in E(G)} x_{ik} & \text{if } i = j,
    \end{cases} \tag{where $i, j \in [n]$}
\end{align*}
where the $x_{ij}$ are independent indeterminates, one for each edge of $G$.

Note that $L_G$ is not of full rank since every row and column sums to $0$ by construction.  However, for $G$ connected, $\rank(L_G)=n-1$. Furthermore, the matrix-tree theorem (see \cite{ch} and \cite{Matousek}) tells us that $\det(L_G(1|1)) = \sum_{T}\prod_{e\in T}x_e$, where the sum is over spanning trees of $G$; this gives a first example of the kind of connection between matrices and polynomials that we work with.  

\begin{example}\label{eg feynman}
  Let $\Psi_G = \sum_{T}\prod_{e\not\in T}x_e$ where the sum is over spanning trees of $G$. $\Psi_G$ can be obtained from $\det(L_G(1|1))$ by inverting the variables and then multiplying by the product of all variables, or by working with the (matroidal) dual.
  
  In the quantum field theory context, we are interested in
  \[
  I_G = \int_{x_e\geq 0} \frac{dx_1\cdots dx_{n-1}}{\Psi_G^2|_{x_n=1}}
  \]
  when it converges.  Integrating $x_1$ and using contraction-deletion properties of $\Psi_G$, we get
  \[
  I_G = \int_{x_e\geq 0} \frac{dx_2\cdots dx_{n-1}}{\Psi_{G/1}\Psi_{G\backslash 1}|_{x_n=1}}.
  \]
  Similarly, integrating $x_2$, where we take edge $2$ to share a vertex with edge $1$, gives
  \[
  I_G = \int_{x_e\geq 0} \frac{dx_3\cdots dx_{n-1}}{(\Psi_{G/12}\Psi_{G\backslash 1/2}-\Psi_{G\backslash 2/1}\Psi_{G\backslash 12})|_{x_n=1}}.
  \]
  Here, something special occurs.  The denominator can be rewritten as the square of a new polynomial $\Psi_G^{1,2}$ which cannot be obtained simply from contraction and deletion but which is closely related to the spanning forest polynomials that we define in the next section.  The identity that permits this rewriting comes down to the same Dodgson identity that we opened with.
  \[
  I_G =  \int_{x_e\geq 0} \frac{dx_3\cdots dx_{n-1}}{(\Psi_G^{1,2})^2|_{x_n=1}}
  \]
  can then be integrated obtaining further spanning forest polynomials in the denominator and logarithms in the numerator.  This can always be continued to five integrations, and for suitably nice graphs and edge orders, this continues further.  Quadratic identities for these polynomials continue to be useful for rewriting and simplfying the denominators, which motivates our interest in them.  See \cite{br} and \cite{by} for details.  
\end{example}

\subsection{Statement on AI usage}

The authors did not use AI to prepare this paper, the bulk of the paper's contents having been developed between May and August 2025. We take full responsibility for the correctness of the material presented.

%% file: sf-identities.tex
\section{Spanning forest identities}\label{sec:sf-identities}

In this section, we introduce identities of spanning forest polynomials and formulate this project's driving questions. 

\subsection{Set partitions and spanning forest polynomials}

A {\em partial set partition} of $[m]$ is a collection $A = \{A_1, \ldots, A_k\}$ of disjoint subsets of $[m]$, where each $A_i$ is called a {\em block} of $A$. Given partial set partitions $A$ and $B$ of $[m]$, we say that $B$ {\em extends} $A$ (denoted $B \succcurlyeq A$) if they have the same number of blocks and each block of $B$ contains some block of $A$.
$A$ is a {\em (full) set partition} if it is a partial set partition and the union of its blocks is $[m]$.
The set of all full set partitions of $[m]$ is denoted $\Pi_{m}$, and the set of all full set partitions of $[m]$ with $k$ blocks is denoted $\Pi_{m,k}$.

Following \cite{fy}, it is convenient to notate to notate a partial set partition $A$ of $[m]$ with $k$ blocks by $(a_1, a_2, \ldots, a_{m})$ with $a_i \in [k]\cup\{-\}$, where $a_i=j$ indicates that $i$ is in is in block $j$ for $j\in [k]$ or is not in a block if $j=-$.  Since we work with unordered partial set partitions, permuting the labels in $[k]$ does not change the partial set partition.

\begin{example}
    Here are some examples of (partial) set partitions of $[4]$:
    \begin{align*}
        A &= (1,-,2,-) = \br{\br{1},\br{3}} \\
        B &= (1,-,2,1) = \br{\br{1,4},\br{3}} \\
        C &= (2,-,1,2) = \br{\br{1,4},\br{3}} \\
        D &= (1,2,2,1) = \br{\br{1,4},\br{2,3}}.
    \end{align*}
    Notice that $A \preccurlyeq B = C \preccurlyeq D$, and that $D\in\Pi_{4,2}$ is the only full set partition.
\end{example}

Given a graph $G$, a subgraph $F$ of $G$ is a {\em spanning forest} of $G$ if $V(F) = V(G)$, $E(F) \subseteq E(G)$, and $F$ does not contain any cycles. The connected components of $F$ are called its {\em trees}. A {\em spanning tree} is a spanning forest with one tree. A spanning forest $F$ is determined by its edge set, so when convenient we view $F$ as a subset of $E(G)$, noting that from this perspective, we implicitly have a single vertex tree of $F$ for each vertex of $G$ that is not incident to any edge of $F$. Because of this, if we let $F[E(G)]$ be the polynomial ring with independent indeterinates $x_e$ indexed by the edges $e$ of $G$, then each spanning forest (resp. spanning tree) $F$ of $G$ is determined by its {\em spanning forest monomial} (resp. {\em spanning tree monomial}) $\prod_{e \in F} x_e$ in $F[E(G)]$.

\begin{example}
    If $G$ is the complete graph on $[4]$, then $x_{12} x_{34}$ is a spanning forest monomial and $x_{12} x_{13} x_{34}$ is a spanning tree monomial.
\end{example}

Given a partial set partition $A$ of $[m]$ with $k$ blocks, we say a spanning forest $F$ of $G$ {\em agrees} with $A$ if $F$ has $k$ trees, each of which has a vertex set containing some block of $A$, and we define the {\em spanning forest polynomial of $A$} in $G$ to be the sum of all spanning forest monomials of spanning forests in $G$ agreeing with $A$. The {\em spanning tree polynomial} of $G$ (or the {\em Kirchhoff polynomial} of $G$) is the spanning forest polynomial of $(1,1,\ldots, 1)$ in $G$.  In Example~\ref{eg feynman} we worked with a dual version of the spanning tree polynomial where the monomials were built of the edges not in the tree instead of the edges in the tree.  This is because the physics naturally gives us this dual version of the polynomial, while mathematically the spanning tree polynomial itself is more natural.  It is easy to convert between them.  The spanning forest polynomial analogue mentioned in Example~\ref{eg feynman} is similarly dual to those we work with here.
Notice that given a partial set partition $A$ of $[m]$, the spanning forest polynomial of $A$ is the sum of the spanning forest polynomials of the full set partitions of $[m]$ extending $A$, which itself is the sum of the spanning forest polynomials of the full set partitions of $[n]$ extending $A$.

\begin{example}\label{eg:sf-identities:sfpoly}
    The spanning forest polynomial of $(1,-,-,2)$ in $K_4$ is
    \begin{align*}
        x_{12} x_{13} + x_{12} x_{23} + x_{13} x_{23} + x_{12} x_{34} + x_{13} x_{24} + x_{23} x_{24} + x_{23} x_{34} + x_{24} x_{34}.
    \end{align*}
    This is equal to the sum of the spanning forest polynomials of the four full set partitions that extend $(1,-,-,2)$ as illustrated below and giving a particular case of the observation immediately before this example. 
    \begin{center}
        \includegraphics[width=\linewidth]{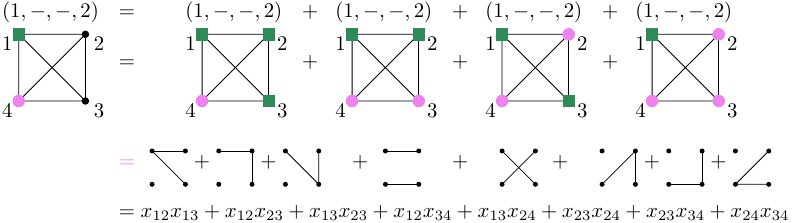}
    \end{center}
\end{example}

\subsection{Spanning forest identities}\label{subsec spanning forest identities}

We would now like to consider identities between these spanning forest polynomials, with some emphasis on quadratic identities.

Let $F[\Pi_m]$ be the polynomial ring in which the set partitions of $[m]$ are independent indeterminates. Then given a graph $G$ on $[n]$ with marked vertices $[m] \subseteq [n]$, we define the {\em interpretation map of set partitions of $[m]$ in $G$} to be the $F$-algebra homomorphism $\Phi_m^{(G)} : F[\Pi_m] \to F[E(G)]$ that maps a set partition $A$ to its spanning forest polynomial in $G$.  (We will not return to using $\Phi$ for the polynomials in Example~\ref{eg feynman}, so this reuse of notation should cause no confusion.)

\begin{example}\label{eg dodgson as set partitions}

  It turns out, see Example~\ref{eg:mtt-det-identities:delta}, that the Dodgson identity described in the introduction when applied to a suitable version of the graph Laplacian gives us that for any graph $G$ with three marked vertices 

\begin{align*}
    \Phi_3^{(G)}(1,1,1)\Phi_3^{(G)}(1,2,3) &= \Phi_3^{(G)}(-,1,2)\Phi_3^{(G)}(1,2,-) - \Phi_3^{(G)}(1,2,1)\Phi_3^{(G)}(1,2,1).
\end{align*}
We call this a \emph{quadratic spanning forest identity} since it is quadratic in the $\Phi$s.

In particular, if $G=K_3$ then this identity becomes
\begin{align*}
    (x_{12} x_{13} + x_{12} x_{23} + x_{13} x_{23}) (1) &= (x_{12} + x_{13})(x_{13} + x_{23}) - (x_{13})(x_{13}),
\end{align*}
which the reader can readily verify is true in $F[E(K_3)]$.  However, we are particularly interested in identities which, like this one, are true for all graphs with the required number of marked vertices, and so are general identities, rather than being specific to particular graphs.

By rewriting the partial set partitions in terms of the full set partitions that extend them, it suffices consider identities that only use full set partitions.  In this case, our identity becomes
\begin{align*}
  & \Phi_3^{(G)}(1,1,1)\Phi_3^{(G)}(1,2,3) \\
  & =  \left(\Phi_3^{(G)}(1,1,2) + \Phi_3^{(G)}(1,2,1)\right)\left(\Phi_3^{(G)}(1,2,1) + \Phi_3^{(G)}(1,2,2)\right) - \Phi_3^{(G)}(1,2,1)\Phi_3^{(G)}(1,2,1) \\
  & = \Phi_3^{(G)}(1,1,2)\Phi_3^{(G)}(1,2,1) + \Phi_3^{(G)}(1,1,2) \Phi_3^{(G)}(1,2,2) + \Phi_3^{(G)}(1,2,1)\Phi_3^{(G)}(1,2,2).
\end{align*}

Up to the symmetries of the problem this is the only general quadratic spanning forest polynomial identity for $m=3$ (see Section 3 of \cite{fy}).

For $m=4$ one can generalize this question to ask about identities relating $\Phi_4^{(G)}(1,1,1,1)\Phi_4^{(G)}(1,2,3,4)$ to linear combinations of terms for the form $\Phi_4^{(G)}(A)\Phi_4^{(G)}(B)$ where $A\in\Pi_{4,3}$ and $B\in \Pi_{4,2}$.  This was solved by one of us with Aleks Vlasev in \cite{vy} giving an eight dimensional space of solutions.  Since one side of each equation is always $\Phi_4^{(G)}(1,1,1,1)\Phi_4^{(G)}(1,2,3,4)$, it is more convenient to take differences of these equations to obtain identities simply for linear combinations of terms of the form $\Phi_4^{(G)}(A)\Phi_4^{(G)}(B)$ where $A\in\Pi_{4,3}$ and $B\in \Pi_{4,2}$.  This was the setting studied in \cite{fy} and generalized therein to $A\in \Pi_{m,m-1}$, $B\in \Pi_{m,2}$.  From there, we are motivated to consider quadratic identities in which for each term, one factor is a set partition of $[m]$ with $k$ blocks, and the other is a set partition of $[m]$ with $\ell$ blocks.

\end{example}

With this notation and motivation in mind, 
we call $X_m^{(G)} := \ker \Phi_m^{(G)}$ the {\em space of spanning forest identities on $G$}. In the case where $G = K_n$, we write $X_m^{(n)}$ and $\Phi_m^{(n)}$ for $X_m^{(K_n)}$ and $\Phi_m^{(K_n)}$.  We lose no generality by working with complete graphs; to recover a non-complete graph on the same vertex, simply set the variables for the edges that are not present to $0$.  By that same reasoning, we have that $X_m^{(n)} \subseteq X_m^{(G)}$ for any $G$ on $[n]$ with marked vertices $[m]$.
The space of all spanning forest identities which are general for graphs with marked vertices $[m]$ is thus
\begin{align*}
    X_m &:= \bigcap_{G} X_m^{(G)} = \bigcap_{n \ge m} X_m^{(n)}.
\end{align*}

We can decompose the act of interpreting a partial set partition of $[m]$ in $K_n$ into two processes: extending it to a formal sum of full set partitions of $[n]$, and interpreting the terms of this formal sum as spanning forest polynomials in $K_n$.  That is, $\Phi_m^{(n)} = \phi_n \circ \iota_m^{(n)}$, where we use the notation $\phi_n = \Phi_n^{(n)}$ and where $\iota_m^{(n)} : F[\Pi_m] \to F[\Pi_n]$ is the {\em extension map of set partitions} defined by sending a set partition $A \in \Pi_m$ to $\sum_{\tilde{A} \succcurlyeq A} \tilde{A}$, where this sum is taken over all $\tilde{A} \in \Pi_n$ extending $A$ viewed as a partial set partition of $[n]$.

\begin{example}
    We can interpret the set partition $(1,2)$ of $[2]$ as a spanning forest polynomial in $K_4$ as follows:
    \begin{align*}
        \Phi_2^{(4)}(1,2) &= \phi_4(1,2,1,1) + \phi_4(1,2,1,2) + \phi_4(1,2,2,1) + \phi_4(1,2,2,2) \\
                            &= (x_{13} x_{14} + x_{13} x_{34} + x_{14} x_{34}) + (x_{13} x_{24}) + (x_{14} x_{23}) + (x_{23} x_{24} + x_{23} x_{34} + x_{24} x_{34}).
    \end{align*}
\end{example}

As we are interested in keeping track of the number of parts in the partitions, for each multiset $M$ of elements of $[m]$, define
\begin{align*}
  \mc{A}_{m,M} &:= \spann_F\pr{\prod_{a\in M}\Pi_{m,a}},
\end{align*}
where we are using the pointwise product of sets $ST = \{st: s\in S, t\in T\}$.
Likewise define $X_{m,M}=X_m \cap \mc{A}_{m,M}$.  Note that this does make $F[\Pi_m]$ into a multigraded ring, but our maps are not graded nor is $X_{m,M}$ a graded ideal.  It is true, though, that $X_{m,M} = \bigcap_{n \ge m} \ker \Phi_{m,M}^{(n)}$. When $|M|$ is small we write the numbers of parts directly as subscripts (e.g., $\mathcal{A}_{m,k,\ell}$).

The motivating question, which we only take small steps towards, is the following.
\begin{question}\label{q:sf-identities:dim}
  Given parameters $(m,M)$, what is $\dim X_{m,M}$?
\end{question}

What was already known is the following:
\begin{theorem}[Fraser, Yeats '23; \cite{fy}]
    $\dim X_{m,2,m-1} = m(m-2)$ for all $m \ge 4$.
\end{theorem}
In Example~\ref{eg dodgson as set partitions}, we saw just one identity relating set partitions with different numbers of parts and hence no identities within $X_{3,2,1}$ when $m=3$ rather than $3\cdot 1 = 3$ as the pattern in the theorem would suggest if it held for $m=3$.  The reason for this is that $m=3$ case has extra symmetries, and so what should be separate possibilities all wind up agreeing, see \cite{fy} Section 3.  The conventions of \cite{fy} differed from ours, in particular taking an ordered list of part sizes rather than a multiset, and with the conventions of that paper the theorem does hold for $m=3$ as well.

\subsection{Forest identity-genericity}\label{subsec genericity}

Fix parameters $(m,k,\ell)$. Notice that for $n \ge m$, $\dim X_{m,k,\ell}^{(n)}$ is the row nullity of a matrix which has for each $(A,B) \in \Pi_{m,k} \times \Pi_{m,\ell}$ a row containing the coefficients of the various edge monomials in $\Phi_m^{(n)}(AB)$.  In unpublished work of one of us with Melanie Fraser \cite{fynotes}, we computed these dimensions when $n=m$ in the cases given on the left hand side in Table~\ref{table:id-dim-table}.  At that time, we had expected that working with $K_m$ when we have $m$ marked vertices would be sufficiently generic.  However, as we have extended the calculations, shown on the right hand side of Table~\ref{table:id-dim-table}, we see that this is not true.   Some notation will help us discuss this further.

\input{id-dim-table.tex}

Taking care of how spanning forests of $K_n$ relate to spanning forests of $K_{n+1}$, the following proposition implies that $(X_{m,k,\ell}^{(n)})_{n \ge m}$ is a decreasing chain of subspaces and hence that $\dim X_{m,k,\ell} = \inf_{n \ge m} \dim X_{m,k,\ell}^{(n)}$.

\begin{proposition}\label{prop containment of Xs}
    $X_{m,M}^{(n)} \supseteq X_{m,M}^{(n+1)}$ for all multisets $M$ of elements of $[m]$ and all $n \ge m$.
\end{proposition}
\begin{proof}
    Define the map
    \begin{align*}
        \theta : F[E(K_{n+1})] &\to F[E(K_n)][t] \\
        x_{i,j} &\mapsto x_{i,j} \tag{for $i,j \in [n]$ distinct} \\
        x_{i,n+1} &\mapsto t \tag{for $i \in [n]$}
    \end{align*}
    Given some $A \in \Pi_m$, consider the constant and linear terms of $\theta(\Phi_m^{(n+1)}(A))$ with respect to $t$. Since every spanning forest of $K_{n+1}$ agreeing with $A$ contains at least one edge incident to vertex $n+1$, the constant term is zero. Furthermore, notice that there is an $n$-to-1 surjective map between the spanning forests of $K_{n+1}$ agreeing with $A$ with exactly one edge incident to $n+1$ and the spanning forests of $K_n$ agreeing with $A$, given by removing the unique edge incident to $n+1$. Because of this, the coefficient of the linear term of $\theta(\Phi_m^{(n+1)}(A))$ is $n\Phi_m^{(n)}(A)$. Hence, letting $d := \ab{M}$, for any monomial $A^{(1)} \ldots A^{(d)}$ in $\mc{A}_{m,M}$, the coefficient of the degree-$d$ term of $\theta(\Phi_m^{(n+1)}(A^{(1)} \ldots A^{(d)}))$ is
    \begin{align*}
        [t^{d}] \theta\pr{\Phi_{m}^{(n+1)}(A^{(1)} \ldots A^{(d)})}
        = [t^1] \theta\pr{\Phi_{m}^{(n+1)}(A^{(1)})} \ldots [t^1] \theta\pr{\Phi_{m}^{(n+1)}(A^{(d)})}
        = n^d \Phi_{m}^{(n)}(A^{(1)} \ldots A^{(d)}).
    \end{align*}
    Thus, for any $R \in \ker \Phi_{m,M}^{(n+1)}$, since $R$ is a linear combination of these monomials, we have that
    \begin{align*}
        0
        = [t^d] \theta\pr{\Phi_{m,M}^{(n+1)}(R)}
        = n^d \Phi_{m,M}^{(n)}(R),
    \end{align*}
    so $R \in \ker \Phi_{m,M}^{(n)}$. Thus, $\ker \Phi_{m,M}^{(n)} \supseteq \ker \Phi_{m,M}^{(n+1)}$.
\end{proof}

Since $\dim X_{m,k,\ell}^{(n)}$ is a nonnegative integer for all $n \ge m$, there must exist some point at which the descending chain $(X_{m,k,\ell}^{(n)})_{n \ge m}$ stabilizes, suggesting the following notion.

\begin{definition}
    Given parameters $(m,k,\ell)$ and $n \ge m$, we say that $K_n$ is {\em forest identity-generic} with respect to our parameters if for all $n' \ge n$,
    \begin{align*}
        X_{m,k,\ell}^{(n)} &= X_{m,k,\ell}^{(n')}.
    \end{align*}
\end{definition}

Intuitively if $K_n$ is forest identity-generic with respect to parameters $(m,k,\ell)$ then every quadratic spanning forest identity with these parameters on $K_n$ is true for all $K_{n'}$ with $n'\geq n$, and hence by the proposition, also for all $G$ with $m$ marked vertices; that is, $K_n$ has no spurious identities.  If we knew what value of $n$ would give us forest identity genericity, then questions about these identities become finite calculations.

\begin{question}
    Given parameters $(m,k,\ell)$, what is the minimal value of $n$ for which $K_n$ is forest identity-generic?
\end{question}

Contrary to the original expectations of \cite{fynotes}, $K_m$ is not forest identity-generic for all $(m,k,\ell)$, as the bold blue numbers in Table \ref{table:id-dim-table} show.  Concretely, we have the following first example of this phenomenon.

\begin{example}\label{ex:sf-identities:5-3-4-not-generic}
    Fix parameters $(m,k,\ell) = (5,3,4)$ and define the set partitions
    \begin{align*}
        A_{21} &:= (1,2,1,2,3) \\
        A_{22} &:= (1,2,2,2,3) \\
        A_{23} &:= (1,1,2,2,3) \\
        A_{24} &:= (1,1,1,2,3) \\
        B_{5} &:= (1,2,3,2,4) \\
        B_{8} &:= (1,2,2,3,4) \\
        B_{9} &:= (1,1,2,3,4).
    \end{align*}
    Then if we let
    \begin{align*}
        R := A_{21} B_8 + A_{21} B_9 + A_{22} B_9 - A_{23} B_5 - A_{23} B_8 - A_{24} B_5,
    \end{align*}
    we have that $\phi_{5,3,4}(R) = 0$ but
    \begin{align*}
        \Phi_{5,3,4}^{(6)}(R) &= -x_{16} x_{24} x_{26} x_{36} x_{56} + x_{12} x_{26} x_{36} x_{46} x_{56}.
    \end{align*}
    This shows that $X_{5,3,4}^{(5)} \not\subseteq X_{5,3,4}^{(6)}$, so $K_5$ is not forest identity-generic given parameters $(m,k,\ell) = (5,3,4)$. This also shows that $\dim X_{5,3,4}^{(5)}$ (which, according to Table \ref{table:id-dim-table}, is 80) is a strict upper bound for $\dim X_{5,3,4}$.
\end{example}

The result of \cite{fy} quoted above gives that $K_m$ is forest identity-generic for $(m,2,m-1)$ as well as an exact formula for the dimension in that case. There are two quadratic cases where we can prove forest identity-genericity already for $K_m$ and these are the two where there are no identities.

\begin{proposition}\label{prop:sf-identities:generic}
  $\dim X_{m, 2, 2}^{(n)} = \dim X_{m, m-1, m-1}^{(n)} = 0$ for all $n\geq m$ and so in particular $K_m$ is forest identity-generic for $(m, 2, 2)$ and $(m, m-1, m-1)$.   
\end{proposition}

\begin{proof}
  First consider $X_{m, m-1, m-1}^{(m)}$.  A set partition of $[m]$ into $m-1$ parts is a choice of two distinct elements of $[m]$, which are paired in one part, while all other elements of $[m]$ are singleton parts.  The spanning forest polynomial on $K_m$ for such a partition is simply the edge variable for the edge joining the two vertices of the size-two part.  Therefore, the spanning forest polynomials involved in $X_{m, m-1, m-1}^{(m)}$ are algebraically independent; in particular, there are no linear relations between products of pairs of them.  This gives that $\dim X_{m, m-1, m-1}^{(m)} = 0$; Proposition~\ref{prop containment of Xs} then gives that $\dim X_{m, m-1, m-1}^{(n)} = 0$ for all $n\geq m$, and hence the genericity of $K_m$ in this case.

  Next, consider $X_{m, 2, 2}^{(m)}$.  The argument is a bit more involved in this case.  Suppose we had a linear relation
  \begin{equation}\label{eq lin dep}
  \alpha_1 \phi_m(A_1)\phi_m(B_1) + \alpha_2 \phi_m(A_2)\phi_m(B_2) + \cdots + \alpha_k\phi_m(A_k)\phi_m(B_k) = 0
  \end{equation}
  with $A_i, B_i \in \Pi_{m, 2}$.
Note that given an edge monomial appearing in some $\phi_m(A_i)\phi_m(B_i)$, the exponent of each variable is either $0$, $1$, or $2$, and if a variable appears to the power $2$, then the corresponding edge appears in both spanning forests that made up that monomial.  Note also that a spanning forest determines a set partition of the vertices by the partition of the forest into its component trees.

Focus on the first term.
If $A_1=B_1$, then picking one spanning forest contributing to $\phi_m(A_1)$ and squaring it, we obtain a term in $\phi_m(A_1)\phi_m(B_1)$ that, by the observations above, cannot appear in any other term, so taking the coefficient of this monomial in \eqref{eq lin dep} yields $\alpha_1 = 0$.

\begin{figure}
  \begin{center}
    \includegraphics{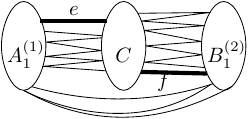} \qquad \includegraphics{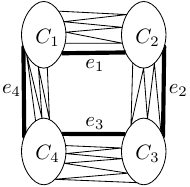}
  \end{center}
  \caption{The set up for the proof in the nested and nonnested cases.}\label{fig nested and nonnested}
\end{figure}

Suppose now that $A_1\neq B_1$ and that the two set partitions are nested in the sense that one part of one of them is strictly contained in one part of the other.  Then write $A_1 = \{A_1^{(1)}, A_1^{(2)}\}$ and $B_1= \{B_1^{(1)}, B_1^{(2)}\}$ with $A_1^{(1)} \subsetneq B_1^{(1)}$ and hence $B_1^{(2)}\subsetneq A_1^{(2)}$.  This gives a partition of $[m]$ into three parts $\{A_1^{(1)}, C = B_1^{(1)}\cap A_1^{(2)}, B_1^{(2)}\}$.  Take three spanning trees $T_1$, $T_2$, and $T_3$ on $A_1^{(1)}$, $C$, and $B_1^{(2)}$ respectively.  Let $e$ be an edge between $A_1^{(1)}$ and $C$, and let $f$ be an edge between $C$ and $B_1^{(2)}$.  This set-up is illustrated on the left in Figure~\ref{fig nested and nonnested}.  Let $\mu$ be the monomial formed of the square of the edges of $T_1$, $T_2$, and $T_3$ along with one copy each of $e$ and $f$.  The monomial $\mu$ appears in $\phi_m(A_1)\phi_m(B_1)$ since the two trees $T_1$ and $T_2\cup\{f\} \cup T_3$ form a spanning forest associated to $A_1$ and the two trees $T_1\cup\{e\}\cup T_2$ and $T_3$ form a spanning forest associated to $B_1$.

Next, we claim that $\mu$ appears in no other term in \eqref{eq lin dep}.  Suppose on the contrary that it did, say in  $\phi_m(A_i)\phi_m(B_i)$ for some $i\neq 1$.  Since the trees $T_1$, $T_2$, and $T_3$ must be subgraphs of some spanning forest of $A_i$ and of some spanning forest of $B_i$, and since the components of the forest determine the partition, we see that the only possibilities for $A_i$ and $B_i$ are $\{A_1^{(1)}\cup C, B_1^{(2)}\}$, $\{A_1^{(1)}\cup B_1^{(2)}, C\}$, and $\{A_1^{(1)}, C\cup B_1^{(2)}\}$.  We already showed that if $A_i=B_i$ then the coefficient must be $0$, so we must take two of these set partitions.  If we take the first two, then $\{A_i, B_i\} = \{A_1, B_1\}$, which is a contradiction.  However, if we take the third one, then neither $e$ nor $f$ can appear in a spanning forest polynomial for this partition, and so we cannot build a spanning forest polynomial for this partition from $\mu$, which is also a contradiction.  Therefore $\mu$ appears only in the first term, and so taking the coefficient of $\mu$ in \eqref{eq lin dep}, we obtain $\alpha_1=0$.

Finally, suppose $A_1\neq B_1$ and the set partitions are not nested.  Again write $A_1 = \{A_1^{(1)}, A_1^{(2)}\}$ and $B_1= \{B_1^{(1)}, B_1^{(2)}\}$ and let $C_1 = A_1^{(1)}\cap B_1^{(1)}$, $C_2 = A_1^{(1)}\cap B_1^{(2)}$, $C_3 = A_1^{(2)}\cap B_1^{(2)}$ and $C_4 = A_1^{(2)}\cap B_1^{(1)}$, all of which are nonempty by hypothesis.  Let $T_1$, $T_2$, $T_3$, and $T_4$ be spanning trees of $C_1$, $C_2$, $C_3$, and $C_4$ respectively.  Let $e_1$ be an edge between $C_1$ and $C_2$, and likewise for $e_2$, $e_3$, and $e_4$ (in particular, $e_4$ is between $C_4$ and $C_1$).  This set-up is illustrated on the right in Figure~\ref{fig nested and nonnested}. Let $\mu$ be the monomial consisting of $e_1e_2e_3e_4$ and the squares of the edge variables of $T_1$, $T_2$, $T_3$, and $T_4$.

We claim that $\mu$ cannot appear in any other term.  Suppose on the contrary that it appeared in $\phi_m(A_i)\phi_m(B_i)$ for some $i\neq 1$.  As above, that would mean that the parts of the partitions $A_i$ and $B_i$ would need to be formed from unions of $C_1$, $C_2$, $C_3$, and $C_4$. Furthermore, we would need to be able to use edges from $e_1$, $e_2$, $e_3$, and $e_4$ to link up the sets which are unioned into any part, and exactly two of these edges need to be used to link the parts of $A_i$ and the other two to link the parts of $B_i$.  Other than the pair $A_1 B_1$, the only other possibilities for such partitions are $\{C_1, C_2\cup C_3\cup C_4\}$ with $\{C_1\cup C_2\cup C_4, C_3\}$ and $\{C_2, C_1\cup C_3\cup C_4\}$ with $\{C_1\cup C_2\cup C_3, C_4\}$.  However, both of these are nested pairs of set partitions and so by the arguments above have coefficient $0$.  Then once more taking the coefficient of $\mu$ in \eqref{eq lin dep}, we obtain $\alpha_1=0$.

The same argument holds for each term, and hence the linear identity is trivial, and so $\dim X_{m, 2, 2}^{(m)} = 0$. Proposition~\ref{prop containment of Xs} then gives $\dim X_{m, 2, 2}^{(n)} = 0$ for all $n\geq m$ and hence the genericity of $K_m$ in this case.

\end{proof}

%% file: id-dim-table.tex
\begin{table}
    \setlength\tabcolsep{4pt}
    \begin{subtable}[b]{0.485\linewidth}
        \centering
        \begin{tabular}{c|cc}
            $m=n=4$ & $k=2$ & $k=3$ \\
            \hline
            $\ell=2$ & 0 & \\
            $\ell=3$ & 8 & 0
        \end{tabular}
        \caption{4 marked vertices on $K_4$}
        \label{table:id-dim-table:m4n4}
    \end{subtable}
    \begin{subtable}[b]{0.485\linewidth}
      \phantom{DUMMY}
     \end{subtable}
    \bigskip
    \begin{subtable}[b]{0.485\linewidth}
        \flushright
        \centering
        \begin{tabular}{c|ccc}
            $m=n=5$ & $k=2$ & $k=3$ & $k=4$\\
          \hline
          $\ell=2$ & 0 & &\\
          $\ell=3$ & 68 & 20 &\\
          $\ell=4$ & 15 & 80 & 0
        \end{tabular}
        \caption{5 marked vertices on $K_5$}
        \label{table:id-dim-table:m5n5}
    \end{subtable}
    \hfill
    \begin{subtable}[b]{0.485\linewidth}
        \flushright
        \centering
        \begin{tabular}{c|ccc}
            $m=5$, $n=6$ & $k=2$ & $k=3$ & $k=4$\\
            \hline
            $\ell=2$ & 0 & & \\
            $\ell=3$ & {\color{blue}\bf 64} & 20 &\\
            $\ell=4$ & 15 & {\color{blue}\bf 64} & 0
        \end{tabular}
        \caption{5 marked vertices on $K_6$}
        \label{table:id-dim-table:m5n6}
    \end{subtable}
    \bigskip
    \begin{subtable}[b]{0.485\linewidth}
        \flushright
        \centering
        \begin{tabular}{c|cccc}
            $m=n=6$ & $k=2$ & $k=3$ & $k=4$ & $k=5$\\
            \hline
            $\ell=2$ & 0 & & &\\
            $\ell=3$ & 280 & 210 & & \\
            $\ell=4$ & 180 & 2052 & 340 & \\
            $\ell=5$ & 24 & 295 & 390 & 0
        \end{tabular}
        \caption{6 marked vertices on $K_6$}
        \label{table:id-dim-table:m6n6}
    \end{subtable}
    \hfill
    \begin{subtable}[b]{0.485\linewidth}
        \flushright
        \centering
        \begin{tabular}{c|cccc}
            $m=6$, $n=7$ & $k=2$ & $k=3$ & $k=4$ & $k=5$\\
            \hline
             $\ell=2$ & 0 & & &\\
             $\ell=3$ &    & 210 & & \\
             $\ell=4$ & {\color{blue}\bf 175} & {\color{blue}\bf 1123} & {\color{blue}\bf 210} & \\
            $\ell=5$ & 24 & {\color{blue}\bf 175} & {\color{blue}\bf 280} & 0
        \end{tabular}
        \caption{6 marked vertices on $K_7$}
        \label{table:id-dim-table:m6n7}
    \end{subtable}
    \caption{Values of $\dim X_{m,k,\ell}^{(n)}$ given parameters $(m,k,\ell)$ and $n \ge m$}
    \label{table:id-dim-table}
\end{table}

%% file: mtt-det-identities.tex
\section{Connections between determinantal identities and spanning forest identities}

As mentioned in the introduction, spanning forest identities can be derived from determinantal identities. In particular, by the all-minors matrix-tree theorem, the determinantal minors of the graph Laplacian $L_G$ enumerate certain families of spanning forests polynomials of $G$. 

In Example \ref{ex:sf-identities:5-3-4-not-generic}, we saw that for some parameters $(m,k,\ell)$, $K_m$ can fail to be forest identity-generic. However, this failure does not occur in the context of determinantal identities: Muir's law of extensible minors states that, as we will make precise shortly, the general homogeneous identities involving marked rows and columns from $[m]$ are exactly those that are true on the generic $m$-by-$m$ matrix.

Muir's law of extensible minors is usually stated for generic (i.e. whose entries are independent indeterminates) square matrices; however, it also applies to generic symmetric square matrices and to Laplacians of complete graphs, which is the context where it will be useful for us.  

Thus the all-minors matrix-tree theorem together with Muir's law says that to create general spanning forest identities involving marked vertices $[m]$, one can simply convert determinantal identities of the Laplacian of $K_m$. This is the subject of Section \ref{subsec mtt}.
One can then ask to go the other way: for which parameters can {\em all} general spanning forest identities be derived from some determinantal identity? In Section \ref{subsec:mtt-det-identities:genericity}, we discuss some sufficient conditions for this to occur that pertain to the \emph{MTT maps}; these maps will be studied in more detail in Section \ref{sec:narayana}.

\subsection{Determinantal identities and Muir's law}\label{subsec det and Muir}

To prepare for the statement of Muir's law, we introduce determinantal identities of complete graph Laplacians just as we introduced spanning forest identities of complete graphs in Section \ref{subsec spanning forest identities}. Our treatment of determinantal identities is adapted from \cite{bs}.

We will describe minors of $L_G$ with minor symbols. A {\em minor symbol} on $[m]$ of size $k$ is a formal symbol $\delta(\alpha \mid \beta)$, where $\alpha, \beta \subseteq [m]$ are of size $k$.  To keep notation light, whenever unambiguous we will omit curly braces and commas in writing the $\alpha$ and $\beta$. 
We let $\Delta_{m,k}$ denote the collection of all minor symbols on $[m]$ of size $k$ and $\Delta_m := \bigcup_{k=0}^m \Delta_{m,k}$, and we will work with the polynomial ring $F[\Delta_m]$.

The {\em interpretation map on minor symbols on $[m]$} 
is the $F$-algebra homomorphism $\psi_m : F[\Delta_m] \to F[E(K_m)]$ determined by sending $\delta(\alpha \mid \beta)$ to $\det L_{K_m}(\alpha \mid \beta)$, which is the determinant of the submatrix of $L_{K_m}$ obtained by removing rows $\alpha$ and columns $\beta$. As convention, we declare that $\det L_{K_m} ([m] \mid [m]) = 1$. We call $Y_m := \ker \psi_m$ the {\em space of determinantal identities of $L_{K_m}$ with marked rows/columns $[m]$}.

To interpret minor symbols on $[m]$ in $K_n$ for $n \ge m$, we have the injective $F$-algebra map $\mu_m^{(n)} : F[\Delta_m] \to F[\Delta_n]$ sending each minor symbol on $[m]$ to itself but viewed as a minor symbol on $[n]$, which we call the {\em extension map of minor symbols} or the {\em Muir map}. Define $\Psi_m^{(n)} := \psi_n \circ \mu_m^{(n)}$, and call $Y_m^{(n)} := \ker \Psi_m^{(n)}$ the {\em space of determinantal identities of $L_{K_n}$ with marked rows/columns $[m]$}. As a shorthand, we let $Y_m := Y_m^{(m)}$.

$F[\Delta_m]$ is graded by size, where the size of a monomial of minor symbols is the sum of the sizes of the symbols and can also be multigraded by the multiset of symbol sizes.
Analogously to Proposition~\ref{prop containment of Xs}, $(Y_{m,M}^{(n)})_{n \ge m}$ is a decreasing chain of subspaces, where the $M$ subscript indicates taking only those elements with multiset of symbol sizes of $M$; however, we already have more than this from classical results on determinantal identities: Muir's law of extensible minors implies that the chain of homogeneous subspaces stabilizes immediately.
\begin{theorem}[Muir's law of extensible minors for Laplacians of complete graphs]\label{thm:mtt-det-identities:muir}
    For all $n \ge m$, the homogeneous elements of $Y_{m}$ (with respect to the $\mb{Z}_{\ge 0}$-grading on $F[\Delta_m]$) are contained in $Y_{m}^{(n)}$.
\end{theorem}

In other words, if a homogeneous identity holds on the Laplacian of $K_m$, it is an identity regardless of the Laplacian in which one interprets it. Hence, for all $n \ge m$ and multisets $M$, $Y_{m,M}^{(n)} = Y_{m,M}$.
The usual version of Muir's law is for generic matrices: that is, matrices with independent indeterminate entries.  The complete Laplacian version does not immediately follow as we need all homogeneous minor identities of complete Laplacians to extend, not just those which are also generic matrix identities.  However, the same proof technique works for this case.  For completeness, we give the argument in Appendix~\ref{header:muir-pf} based on the proof of \cite{bs}.

\begin{example}
    Notice that for any $2\times 2$ matrix, not just a Laplacian matrix, 
    \begin{align*}
        \delta(\emptyset \mid \emptyset) - \delta(1 \mid 1) \delta(2 \mid 2) - \delta(1 \mid 2) \delta(2 \mid 1)
    \end{align*}
    is in the kernel of the interpretation map since it is the definition of the determinant of a 2-by-2 matrix. To homogenize this identity, we can pad out the first term with an extra factor of $\delta(12 \mid 12)$ to get
    \begin{align*}
        \delta(\emptyset \mid \emptyset) \delta(12 \mid 12) - \delta(1 \mid 1) \delta(2 \mid 2) - \delta(1 \mid 2) \delta(2 \mid 1),
    \end{align*}
    which, by Muir's law, is in the kernel of the interpretation map for any square matrix. This is the identity we opened the paper with, often known as the Dodgson identity.
\end{example}

\begin{example}\label{eg:mtt-det-identities:delta}
  Returning to the Laplacian case, since the row and column sums of a graph Laplacian are always zero, the graph Laplacian is never of full rank.  Consequently, the interesting version of the Dodgson identity in this case is not on the Laplacian itself, but on the Laplacian with one row and column removed, such as
    \begin{align*}
        \delta(1 \mid 1) \delta(123 \mid 123) - \delta(12 \mid 12) \delta(13 \mid 13) + \delta(12 \mid 13) \delta(13 \mid 12) \in F[\Delta_3],
    \end{align*}
    which by the previous example and Muir's law is a true identity when interpreted on any graph Laplacian.
\end{example}

\subsection{The all-minors matrix-tree theorem}\label{subsec mtt}

We now discuss the all-minors matrix-tree theorem (see \cite{ch}), which says that the minors of the Laplacian of a graph $G$ enumerate a certain family of spanning forest polynomials of $G$. From our perspective the all-minors matrix-tree theorem associates to each minor symbol in $\Delta_m$ a signed sum of set partitions in a way that is compatible with our determinant map $\psi_m$ and our set partition interpretation map $\phi_m$. That is, we have a map $\tau_m : F[\Delta_m] \to F[\Pi_m]$ called the  {\em MTT map} which satisfies $\psi_m = \phi_m \circ \tau_m$. The goal of this section is to define this map.

Given equicardinal sets $\alpha, \beta \subseteq [m]$, a {\em matching on $\delta(\alpha \mid \beta)$} is a bijection $\sigma : \alpha \to \beta$; it is an {\em MTT matching} if it fixes $\alpha \cap \beta$ pointwise. We write $\sigma \sim \delta(\alpha \mid \beta)$ to indicate that $\sigma$ is an MTT matching on $\delta(\alpha \mid \beta)$. Notice that a matching $\sigma$ on $\delta(\alpha \mid \beta)$ can be viewed as the partial set partition $A_\sigma$ of $[m]$ with a block $\{c\}$ for each $c \in \alpha \cap \beta$ and a block $\{a, \sigma(a)\}$ for each $a \in \alpha \setminus \beta$. In fact, if we are given $\delta(\alpha \mid \beta)$ and $A_\sigma$, we can recover what the matching $\sigma$ is. Hence, we will denote this partial set partition $A_\sigma$ simply as $\sigma$. We say that a set partition $A \in \Pi_{n,k}$ {\em grows} from a minor symbol $\delta(\alpha \mid \beta) \in \Delta_{n,k}$ if it extends a matching $\sigma$ (viewed as a partial set partition) on $\delta(\alpha \mid \beta)$.

\begin{example}
    There are six set partitions that grow from $\delta(124 \mid 135) \in \Delta_{6,3}$:
    \begin{align*}
        (1,2,2,3,3,1), 
        (1,2,2,3,3,2), 
        (1,2,2,3,3,3), \\
        (1,2,3,3,2,1), 
        (1,2,3,3,2,2), 
        (1,2,3,3,2,3).
    \end{align*}
Note that there are two matchings, hence two partial set partitions, each of which completes to three set partitions.
\end{example}

The all-minors matrix-tree theorem involves some slightly annoying signs.
Given a subset $A$ of a linearly-ordered set $S$, an {\em inversion of $A$ in $S$} is a pair $(i,j) \in A \times (S \setminus A)$ such that $j < i$. The number of inversions of $A$ in $S$ is denoted $\eta(A,S)$, and the {\em sign of $A$ in $S$} is $\eps(A,S) := (-1)^{\eta(A,S)}$.
Given a matching $\sigma : \alpha \to \beta$ (where $\alpha, \beta \subseteq [m]$), an {\em inversion of $\sigma$} is a pair $(i,j) \in \alpha \times \beta$ such that $i < j$ and $\sigma(i) > \sigma(j)$. The number of inversions of $\sigma$ is denoted $\eta(\sigma)$, and the {\em sign of $\sigma$} is $\eps(\sigma) := (-1)^{\eta(\sigma)}$.

\begin{example}
    The subset $\{2,4,5\}$ of $[5] = \{1,2,3,4,5\}$ has five inversions, so its sign is negative. The matching $\sigma : \br{1,3,4} \to \br{2,3,4}$ sending $1 \mapsto 5$, $3 \mapsto 3$, and $4 \mapsto 2$ has three inversions, so its sign is also negative.
\end{example}

With this, we define the 
{\em $m$-th MTT map} to be the $F$-algebra homomorphism $\tau_m : F[\Delta_m] \to F[\Pi_m]$ determined by
    \begin{align*}
        \delta(\alpha \mid \beta) &\mapsto \eps(\alpha, [m]) \eps(\beta, [m]) \sum_{\sigma \sim \delta(\alpha \mid \beta)} \eps(\sigma) \sum_{A \preccurlyeq \sigma} A,
    \end{align*}
    where $\ab{\alpha} = \ab{\beta} = k$, the first sum is over all MTT matchings $\sigma : \alpha \to \beta$, and the second sum is over all set partitions $A$ extending $\sigma$ (viewed as a partial set partition).

\begin{theorem}[All-minors matrix-tree theorem]\label{thm:mtt-det-identities:ammtt}
    For all $m \ge 1$, $\psi_m = \phi_m \circ \tau_m$.
\end{theorem}
For a combinatorial proof see \cite{ch}.

\begin{example}
    Notice that
    \begin{align*}
        \det L_{K_4}(13 \mid 23)
        = \det \begin{bmatrix}
            x_{12} & x_{24} \\
            x_{14} & -x_{14} - x_{24} - x_{34}
        \end{bmatrix}
        = -x_{12}x_{14} - x_{12}x_{24} - x_{12}x_{34} - x_{14}x_{24}.
    \end{align*}
    To show that this agrees with what the all-minors matrix-tree theorem tells us, notice that $\eps(13, [4]) = -1$, that $\eps(23, [4]) = 1$, and that the only matching $\sigma$ on $\delta(13 \mid 24)$ (which maps $1 \mapsto 2$ and $3 \mapsto 3$) has positive sign. Furthermore, the two set partitions of $[4]$ that extend $\sigma$ are $(1,1,2,1)$ and $(1,1,2,2)$, whose spanning forest polynomials respectively are $x_{12} x_{34}$ and $x_{12} x_{14} + x_{12} x_{24} + x_{14} x_{24}$.
\end{example}

\begin{remark}\label{rmk giant diagram}
    As a consequence of this theorem, we have that the following diagram commutes for all $m$. Indeed, the theorem says that the triangles commute, and the rectangles commute construction.
    \begin{center}
        \begin{tikzcd}
                & F[\Pi_m] \ar[dr, "\phi_{m}"] \ar[dd, hook, "\iota_{m}^{(m+1)}", near end] \\
            F[\Delta_m] \ar[ur, "\tau_{m}", near end] \ar[rr, "\psi_{m}", crossing over, near end] \ar[dd, "\mu_{m}^{(m+1)}", hook, near start]
                &
                    & F[E(K_m)] \\
                & F[\Pi_{m+1}] \ar[dr, "\phi_{m+1}"] \ar[dd, hook, "\iota_{m+1}^{(m+2)}", near end] \\
            F[\Delta_{m+1}] \ar[ur, "\tau_{m+1}", near end] \ar[rr, "\psi_{m+1}", crossing over, near end] \ar[dd, "\mu_{m+1}^{(m+2)}", hook, near start]
                &
                    & F[E(K_{m+1})] \\
                & F[\Pi_{m+2}] \ar[dr, "\phi_{m+2}"] \ar[dd, hook, "\iota_{m+2}^{(m+3)}", near end] \\
            F[\Delta_{m+2}] \ar[ur, "\tau_{m+2}", near end] \ar[rr, "\psi_{m+2}", crossing over, near end] \ar[dd, "\mu_{m+2}^{(m+3)}", hook, near start]
                &
                    & F[E(K_{m+2})] \\
                & \vdots \\
            \vdots
        \end{tikzcd}
    \end{center}
\end{remark}

\begin{example}\label{eg using the framework}
    Applying $\tau_3$ to the determinantal identity
    \begin{align*}
        R := \delta(1 \mid 1) \delta(123 \mid 123) - \delta(12 \mid 12) \delta(13 \mid 13) + \delta(12 \mid 13) \delta(13 \mid 12) \in Y_3
    \end{align*}
    from Example \ref{eg:mtt-det-identities:delta} yields
    \begin{align*}
        S := \tau_3(R) = (1,2,1)^2 - ((1,2,1) + (1,1,2))((1,2,1) + (1,2,2)) + (1,1,1)(1,2,3),
    \end{align*}
    which is the spanning forest identity on $K_3$ given in Example~\ref{eg dodgson as set partitions}. Further, this is an identity on $K_n$ for all $n \ge 3$ since for all $n \ge 3$, $\mu_3^{(n)}(R)$ is an identity on the Laplacian of $K_n$ by Muir's law, so $\tau_n \circ \mu_3^{(n)}(R) = \iota_3^{(n)}(S)$ is a spanning forest identity on $K_n$ by Remark~\ref{rmk giant diagram}.
\end{example}

\subsection{Forest identity-genericity and the MTT map}\label{subsec:mtt-det-identities:genericity}

Example~\ref{eg using the framework} shows how Muir's law can be used to convert a determinantal identity on marked rows/columns $[m]$ into a {\em general} spanning forest identity on marked vertices $[m]$.
Moreover, the example demonstrates how the perspective encapsulated in \ref{rmk giant diagram} gives a clarifying framework for thinking about spanning forest identities.  In \cite{fy} we moved between thinking of spanning forest identities on specific graphs and generic spanning forest identities, between AB pairs (certain pairs of set partitions) as indeterminates and as their interpretations on graphs.  However, we did not have the present framework to cleanly lay out these different pieces and their interrelations, making the arguments more difficult for the reader, and making it harder to think about genericity and in particular leading to overly optimistic expectations on forest identity-genericity as discussed in \ref{subsec genericity}.
 
We ask, then, whether what happens in Example~\ref{eg using the framework} is the only way that general spanning forest identities can be created; that is, if we allow ourselves to freely extend by our extension maps $\iota_m^{(n)}$, then given a spanning forest identity, does there exist some determinantal identity sent to it by the MTT map? 

We can capture this idea with a definition.
    Given parameters $(m,M)$, we say that $\tau_{n,M}$ is {\em full enough} with respect to $m$ if
    \begin{equation}\label{eq:mtt-det-identities:full-enough}
        \iota_{m,M}^{(n)}\pr{\ker \Phi_{m,M}^{(n)}} \subseteq \im \tau_{n,M}.
    \end{equation}

    Intuitively, if $\tau_{n,M}$ is full enough, it is not necessarily surjective, but its image is still large enough to contain our space of quadratic spanning forest identities (after freely extending by $\iota_{m,M}^{(n)}$). One can easily check that $\tau_{n,M}$ being full enough implies that $K_n$ is forest identity-generic with respect to the parameters $m$ and $M$. This is by design, as we have essentially defined the ``laziest'' way to derive forest identity-genericity of $K_n$ from Muir's law.

    This framework leads to the following unanswered questions. 

\begin{question}\label{q:mtt-det-identities:tau-full}
    Given parameters $(m,M)$, what is the minimal value of $n$ for which $\tau_{n,M}$ is full enough?
\end{question}

\begin{question}
    Given parameters $(m,M)$, if $K_n$ is forest identity-generic, is $\tau_{n,M}$ full enough?
\end{question}

Given specific values of $m$, $n$, and $M$, if one expresses $\Phi_{m,M}^{(n)}$, $\iota_{m,M}^{(n)}$, and $\tau_{n,M}$ as matrices with respect to the appropriate monomial bases for their domains and codomains, then checking (\ref{eq:mtt-det-identities:full-enough}) reduces to a matrix computation.
    We have verified via such a computation in SageMath that in the cases $(m,k,\ell) \in \{(5,2,3),(5,3,4)\}$, $\tau_{6,k,\ell}$ is full enough, implying that $K_6$ is forest identity-generic in both cases. In all other cases in which $m \le 5$, we also computed that $\tau_{5,k,\ell}$ is full enough.
    In general, we know that $\tau_{m,M}$ is generic when $k = \ell = 2$ and when $k = \ell = m-1$; these cases are vacuous since no spanning forest identities exist by Proposition \ref{prop:sf-identities:generic}.

%% file: narayana.tex
\section{The ranks of the linear MTT maps are Narayana}\label{sec:narayana}

Now that we have the framework of the previous section, we would like to use it.  While the quadratic case that originally motivated our investigations remains challenging, as a first use of the framework we can investigate the linear MTT maps $\tau_{m,k}$. In particular, we show that the rank of $\tau_{m,k}$ is the {\em $(m,k)$-th Narayana number}, $N(m,k)$, by exhibiting an explicit basis.

\begin{definition}
  A minor symbol $\delta(\alpha \mid \beta) \in \Delta_{m,k}$ is {\em standard} if $1 \in \alpha$, $1 \in \beta$, and $\alpha_i \leq \beta_i$ for all $i\in [k]$ where $\alpha = \{\alpha_1 < \ldots < \alpha_k\}$ and $\beta = \{\beta_1 < \ldots < \beta_k\}$.
  The set of all standard minor symbols in $\Delta_{m,k}$ is denoted by $\tilde{\Delta}_{m,k}$.
\end{definition}

The Narayana numbers count many interesting objects: in particular, the {\em $(m,k)$-th Narayana number}, denoted $N(m,k)$, is the number of Dyck paths with $2m$ steps and $k$ peaks.

\begin{proposition}\label{prop:narayana:count}
    $\ab{\tilde{\Delta}_{m,k}} = N(m,k)$.
\end{proposition}
\begin{proof}
  Note that a Dyck path with $k$ peaks has $k-1$ valleys and is determined by its valleys and endpoint.  We will view our Dyck paths as going from $(1,1)$ to $(m+1, m+1)$ using north and east steps and staying weakly above the diagonal $y=x$.  Then given $\delta(\alpha \mid \beta) \in \tilde{\Delta}_{m,k}$ with  $\alpha = \{\alpha_1  < \ldots < \alpha_k\}$ and $\beta = \{\beta_1 < \ldots < \beta_k\}$, let the valleys of the corresponding Dyck path be $(\alpha_{i+1}, \beta_{i+1})$ for $i = 1, \ldots, k-1$.  This gives the desired bijection.
\end{proof}

\subsection{Two types of MTT relations}\label{subsec:narayana:mtt-relations}

In this section, we identify two families of MTT relations, i.e., elements of $\ker \tau_{m,k}$, which we will use to show that our proposed basis of $\im \tau_{m,k}$ is a spanning set.

The first family of relations comes from the operation of changing one endpoint of a matching while keeping the rest of the matching fixed.  To correctly keep track of the signs, we will need to fix parameters as follows.
\begin{definition}
    Fix $\alpha, \beta' \subseteq [m]$, where $\ab{\alpha} = k$ and $\ab{\beta'} = k+1$. Let $\mb{P}_{\alpha, \beta'}$ be the set of all MTT matchings $\sigma : \alpha \to \beta' \setminus b$, where $b \in \beta'$. We say that $\sigma, \sigma' \in \mb{P}_{\alpha, \beta'}$ {\em differ by a pivot} if they agree at exactly $k-1$ points of $\alpha$. We let the {\em pivot sign} of an element $\sigma : \alpha \to \beta' \setminus b$ of $\mb{P}_{\alpha, \beta'}$ be $(-1)^{\eta(\sigma)+\eta(b, \beta')}$.
\end{definition}

\begin{example}
    Let $\alpha = 1347$ and $\beta' = 12356$.  Let $\sigma$ and $\sigma'$ be as illustrated below. Then $\sigma, \sigma' \in \mb{P}_{\alpha, \beta'}$ differ by a pivot. $\sigma$ has pivot sign $(-1)^{\eta(\sigma)+\eta(5,12356)} = (-1)^{2 + 3} = -1$ and $\sigma'$ has pivot parity $(-1)^{\eta(\sigma') + \eta(2,12356)} = (-1)^{1 + 1} = 1$.
    \begin{center}
        \includegraphics{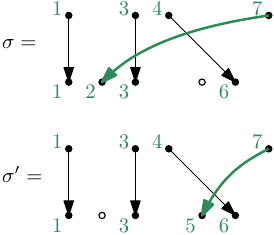}
    \end{center}
\end{example}

We now show that pivots always reverse pivot sign.

\begin{lemma}\label{lem pivot sign}
    If $\sigma, \sigma' \in \mb{P}_{\alpha,\beta'}$ differ by a pivot, then they have opposite pivot sign.
\end{lemma}
\begin{proof}
  Suppose that $\sigma$ and $\sigma'$ agree on $\alpha \setminus a$ and that $\sigma(a)=b$, $\sigma'(a) = b'$.
  Observe that modulo 2, $\eta(\sigma) + \eta(\sigma')$ is the number of crossings created or destroyed when moving from $\sigma$ to $\sigma'$, which can be shown to be the number of elements of $\beta$ strictly between $b$ and $b'$.
 On the other hand, assuming without loss of generality that $b < b'$, $\eta(b, \beta') + \eta(b', \beta')$ is (modulo 2) the number of elements of $\beta'$ strictly less than $b'$ and weakly greater than $b$. Thus, $\eta(\sigma) + \eta(\sigma') + \eta(b,\beta') + \eta(b',\beta')$ is 1 modulo 2, so $\sigma$ and $\sigma'$ have opposite pivot sign.
\end{proof}

\begin{proposition}\label{prop:narayana:pivot-rels}
    Given $\alpha, \beta' \subseteq [m]$ such that $\ab{\alpha} = k$ and $\ab{\beta'} = k+1$,
    \begin{align*}
        \sum_{b \in \beta'} \eps(b,[m]) \eps(b,\beta') \delta(\alpha \mid \beta' \setminus b) \in \ker \tau_{m,k}.
    \end{align*}
\end{proposition}
\begin{proof}
    Fix $\alpha, \beta'$ as in the statement. We want to show
    \begin{align*}
        0 &= \sum_{b \in \beta'} \sum_{\sigma \sim \delta(\alpha \mid \beta' \setminus b)} \sum_{A \succcurlyeq \sigma} \eps(b,[m]) \eps(b,\beta') \eps(\alpha, [m]) \eps(\beta' \setminus b, [m]) \eps(\sigma) A,
    \end{align*}
    which we can do by creating a sign-reversing involution between the terms of this sum. To do this, take a term in the sum.  The information giving that term is a set partition $A \in \Pi_{m,k}$ appearing in this summation and the fact that it grows from some MTT matching $\sigma$ on $\delta(\alpha \mid \beta' \setminus b)$, where $b \in \beta'$. Notice that since $A$ has $k$ blocks, each containing a unique element of $\beta \setminus b$, there exists a unique $b' \in \beta \setminus b$ that shares a block with $b$. Now let $a := \sigma^{-1}(b')$, and let $\sigma' : \alpha \to \beta' \setminus b'$ be the matching that agrees with $\sigma$ on $\alpha \setminus a$ such that $\sigma'(a) = b$; clearly, $\sigma$ and $\sigma'$ differ by a pivot. Furthermore, notice that $A$ grows from $\sigma'$, which is also an MTT matching on $\delta(\alpha \mid \beta' \setminus b')$, where $b' \in \beta'$.  Define our sign reversing involution to map the term of the first instance of $A$ in our summation to the term of this new instance of $A$; this mapping is well-defined by uniqueness of $b'$, and it is an involution since applying it twice performs a pivot on the same edge. To see that it reverses signs, we calculate
    \begin{align*}
        &\epsilon(b,[m])\epsilon(b, \beta')\epsilon(\alpha, [m])\epsilon(\beta' \setminus b, [m])\epsilon(\sigma) 
      \epsilon(b',[m])\epsilon(b', \beta')\epsilon(\alpha, [m])\epsilon(\beta' \setminus b', [m])\epsilon(\sigma')\\
      & = \epsilon(b,[m])\epsilon(b, \beta')\epsilon(\beta' \setminus b, [m])\epsilon(\sigma) 
      \epsilon(b',[m])\epsilon(b', \beta')\epsilon(\beta' \setminus b', [m])\epsilon(\sigma') \\
      & = - \epsilon(b,[m]\epsilon(\beta' \setminus b, [m]) 
      \epsilon(b',[m])\epsilon(\beta' \setminus b', [m]) \tag{by Lemma~\ref{lem pivot sign}} \\
      & = -1
    \end{align*}
    since one can directly show that
    \begin{align*}
        \eta(\beta' \setminus b, [m]) + \eta(\beta' \setminus b', [m])
        &\equiv \eta(b, [m]) + \eta(b', [m])
    \end{align*}
    modulo 2.
\end{proof}

The second family of relations comes from the operation of taking one pairing in a matching and flipping which end belongs to which set.  We again need to fix some parameters to correctly keep track of the signs.
\begin{definition}\label{def:narayana:flip}
    Fix disjoint subsets $\gamma, \alpha, \beta, Q \subseteq [m]$ such that $\ab{\gamma} + \ab{\alpha} + \ab{\beta} \le k-1$ and $\ab{Q} = 2k - 2\ab{\gamma} - \ab{\alpha} - \ab{\beta}$. Let $\mb{F}_{\gamma, \alpha, \beta, Q}$ be the set of all MTT matchings $\sigma : \gamma \alpha R \to \gamma \beta S$, where $R \subseteq Q$ is of size $k - \ab{\gamma} - \ab{\alpha}$ and $S = Q \setminus R$. We say that two elements $\sigma : \gamma \alpha R \to \gamma \beta S$ and $\sigma' : \gamma \alpha R' \to \gamma \beta S'$ of $\mb{F}_{\gamma, \alpha, \beta, Q}$ {\em differ by a flip} if there exist $r \in R \cap S'$ and $s \in S \cap R'$ such that
    \begin{enumerate}
        \item $R \setminus r = R' \setminus s$ and $S \setminus s = S' \setminus r$;
        \item $\sigma$ and $\sigma'$ agree on $\gamma \alpha R \setminus r$;
        \item $\sigma(r) = s$ and $\sigma'(s) = r$.
    \end{enumerate}
    We let the {\em flip sign} of an element $\sigma : \gamma \alpha R \to \gamma \beta S$ of $\mb{F}_{\gamma, \alpha, \beta, Q}$ be $(-1)^{\eta(\sigma) + \eta(\alpha R, \alpha \beta Q)}$.
\end{definition}

\begin{example}
    Let $m = 7$, $k = 4$, $\gamma = 3$, $\alpha = 1$, $\beta = 47$, and $Q = 256$. Let $\sigma$ and $\sigma'$ be as illustrated below. Then $\sigma, \sigma' \in \mb{F}_{\gamma,\alpha, \beta, Q}$ differ by a flip. $\sigma$ has flip sign $(-1)^{\eta(\sigma) + \eta(126, 124567)} = (-1)^{5 + 2} = -1$ and $\sigma'$ has flip sign $(-1)^{\eta(\sigma') + \eta(156, 124567)} = (-1)^{4 + 4}  = 1$.
    \begin{center}
        \includegraphics{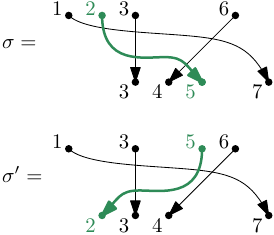}
    \end{center}
\end{example}

\begin{lemma}\label{lem:narayana:flip-parity}
    If $\sigma, \sigma' \in \mb{F}_{\gamma, \alpha, \beta, Q}$ differ by a flip, then they have opposite flip sign.
\end{lemma}
\begin{proof}
    Suppose $\sigma : \gamma \alpha R \to \gamma \beta S$ and $\sigma' : \gamma \alpha R' \to \gamma \beta S'$ differ by a flip, where we let $R, S, R', S', r, s$ be as in Definition \ref{def:narayana:flip}. Assume without loss of generality that $r < s$ (the case where $s < r$ holds symmetrically). As in the proof of Lemma~\ref{lem pivot sign}, we are counting how crossings change, hence the numbers of elements in certain intervals. Specifically, modulo 2, we have
    \begin{align*}
        \eta(\sigma) + \eta(\sigma')
        &\equiv \ab{\alpha \beta Q \cap (r,s)} \\
        \eta(\alpha R, \alpha \beta Q) + \eta(\alpha R', \alpha \beta Q)
        &\equiv \ab{\alpha \beta Q \cap [r,s)},
    \end{align*}
    where $(r,r')$ (resp. $[r,r')$) denotes the open (resp. half-open) interval with endpoints $r$ and $s$. Adding these together, we see that $\sigma$ and $\sigma'$ have opposite flip signs.
\end{proof}

\begin{proposition}\label{prop:narayana:flip-rels}
    Suppose $\gamma, \alpha, \beta, Q \subseteq [m]$ are disjoint such that $\ab{\gamma} + \ab{\alpha} + \ab{\beta} \le k-1$ and $\ab{Q} = 2k - 2\ab{\gamma} - \ab{\alpha} - \ab{\beta}$. Then
    \begin{align*}
        \sum_{Q = R \sqcup S} \eps(\alpha R, \alpha \beta Q) \delta(\gamma \alpha R \mid \gamma \beta S) \in \ker \tau_{m,k},
    \end{align*}
    where the sum is over all subsets $R \subseteq Q$ of size $k - \ab{\gamma} - \ab{\alpha}$ and $S = Q \setminus R$.
\end{proposition}
\begin{proof}
    Fix $\gamma$, $\alpha$, $\beta$, and $Q$ as in the statement. We want to show
    \begin{align*}
        0 = \sum_{Q = R \sqcup S}
        \sum_{\sigma \sim \delta(\gamma \alpha R \mid \gamma \beta Q)}
        \sum_{A \succcurlyeq \sigma}
        \eps(\alpha R, \alpha \beta Q)
        \eps(\gamma \alpha R, [m])
        \eps(\gamma \beta S, [m])
        \eps(\sigma) A,
    \end{align*}
    which we again do by creating a sign-reversing involution between the terms of this sum. To do this, take a term in the sum.  That is, take a set partition $A \in \Pi_{m,k}$ appearing in the summation and the information of an MTT matching $\sigma$ on $\delta(\gamma \alpha R \mid \gamma \beta S)$ which it extends, where $R \subseteq Q$ has size $k - \ab{\gamma} - \ab{\alpha}$ and $S = Q \setminus R$. Since $\ab{\gamma} + \ab{\alpha} + \ab{\beta} \le k-1$, we have that $\ab{Q} \ge k - \ab{\gamma} + 1$, so there must exist some $r \in R$ and $s \in S$ such that $\sigma(r) = s$. Because of this, we can choose $r$ and $s$ so that the edge between them has a minimal endpoint among all edges with both endpoints in $Q$.

    Let $\sigma' : \gamma \alpha R' \to \gamma \beta S'$ differ from $\sigma$ by the flip that swaps $r$ and $s$. Notice that $A$ also grows from $\sigma'$, which one can also show is an MTT matching. We now map the term of our first instance of $A$ in our summation the term of this new instance of $A$; this mapping is an involution since performing it twice will flip the same edge. To see that it reverses sign, calculate
    \begin{align*}
        &\epsilon(\alpha R, \alpha \beta Q)\epsilon(\gamma \alpha R, [m])\epsilon(\gamma \beta S, [m])\epsilon(\sigma) 
      \epsilon(\alpha R', \alpha \beta Q)\epsilon(\gamma \alpha R', [m])\epsilon(\gamma \beta S', [m])\epsilon(\sigma') \\
      & = - \epsilon(\gamma \alpha R, [m])\epsilon(\gamma \beta S, [m])\epsilon(\gamma \alpha R', [m])\epsilon(\gamma \beta S', [m]) \tag{by Lemma~\ref{lem:narayana:flip-parity}} \\
      & = -1
    \end{align*}
    since
    \begin{align*}
        \eta(\gamma \alpha R, [m]) + \eta(\gamma \beta S, [m]) + \eta(\gamma \alpha R', [m]) + \eta(\gamma \beta S', [m])
        \equiv \sum_{x \in \gamma \alpha R} x + \sum_{x \in \gamma \beta S} x + \sum_{x \in \gamma \alpha R'} x + \sum_{x \in \gamma \beta S'} x
        \equiv 0
    \end{align*}
    modulo 2.
\end{proof}

\begin{proposition}\label{prop:narayana:span}
    $\tau_{m,k}(\tilde{\Delta}_{m,k})$ spans $\im \tau_{m,k}$.
\end{proposition}
\begin{proof}
    Throughout, we will implicitly work modulo $\ker \tau_{m,k}$ and prove that (the modulo-$\ker \tau_{m,k}$ image of) $\tilde{\Delta}_{m,k}$ spans $\mc{P}_{m,k} / \ker \tau_{m,k}$.

    Notice that each minor symbol $\delta(\alpha \mid \beta)$ in which $1$ is not included in its row set $\alpha$ or its column set $\beta$ can be expressed as a linear combination of minor symbols whose row and column sets both include 1 (by Proposition \ref{prop:narayana:pivot-rels}). It hence remains to show that each minor symbol $\delta(\alpha \mid \beta)$ such that $1 \in \alpha \cap \beta$ and $\alpha \not\le \beta$ can be expressed as a linear combination of standard minor symbols.

    Towards a contradiction, suppose $\delta(\alpha \mid \beta) \in \Delta_{m,k}$ is a nonstandard minor symbol that cannot be expressed as a linear combination of standard minor symbols modulo $\ker \tau_{m,k}$. Write $\alpha = \{\alpha_1 < \ldots < \alpha_k\}$ and $\beta = \{\beta_1 < \ldots < \beta_k\}$, and choose $\ell$ such that $\alpha_i \le \beta_i$ for $i < \ell$ but $\alpha_\ell > \beta_\ell$. Furthermore, choose $\delta(\alpha \mid \beta)$ such that $\ell$ is maximized. Now define
    \begin{align*}
        R_0 &:= \br{\alpha_\ell, \ldots, \alpha_k} \setminus (\alpha \cap \beta) \\
        S_0 &:= \br{\beta_1, \ldots, \beta_\ell} \setminus (\alpha \cap \beta) \\
        Q &:= R_0 \cup S_0 \\
        \alpha' &:= \alpha \setminus R_0 \\
        \beta' &:= \beta \setminus S_0.
    \end{align*}
    Intuitively, $R_0$ contains all elements in $\alpha$ that are at least $\alpha_\ell$ and $S_0$ contains all elements of $\beta$ that are at most $\beta_\ell$, except that $R_0$ and $S_0$ exclude elements appearing in $\alpha \cap \beta$. Notice by Proposition \ref{prop:narayana:flip-rels} that
    \begin{align*}
        \sum_{Q = R \sqcup S} \pm \delta(\alpha' R \mid \beta' S) \in \ker \tau_{m,k},
    \end{align*}
    where the sum is over all $R$ of size $k-|\gamma|-|\alpha|$ and where the term corresponding to $Q = R_0 \sqcup S_0$ is $\delta(\alpha \mid \beta)$, up to a sign. It remains to show that for all other decompositions $Q = R \sqcup S$, the $i$-th smallest term of $\alpha' R$ is at most the $i$-th smallest term of $\beta' S$ for all $i \le \ell$; from this, it would follow by maximality of $\ell$ that $\delta(\alpha' R \mid \beta' S)$ is a linear combination of standard minor symbols and hence that $\delta(\alpha \mid \beta)$ is also such a linear combination, yielding a contradiction.

    Fix a decomposition $Q = R \sqcup S$ such that $(R,S) \neq (R_0, S_0)$ and fix $i \le \ell$. Notice that $R$ must contain some element of $S_0$. Since all elements of $S_0$ are strictly less than all elements of $R_0$, we must have that the $i$-th smallest element of $\alpha' R$ either is in $\{\alpha_1, \ldots, \alpha_i\}$ or is some element of $S_0$ smaller than $\alpha_i$. In either case, the $i$-th smallest element of $\alpha' R$ is at most $\alpha_i$. Similarly, one can show that the $i$-th smallest element of $\beta' S$ is at least $\beta_i$. Thus, the $i$-th smallest element of $\alpha' R$ is at most the $i$-th smallest element of $\beta' S$, as desired.
\end{proof}

\subsection{Initial partitions}\label{subsec:narayana:init-partitions}

In this section, we will consider the lexicographic ordering of set partitions to assign to each element of $\tau_{m,k}(\Delta_{m,k})$ an {\em initial set partition}. We will then show that the $\tau_{m,k}$-images of no two distinct standard minor symbols have the same initial set partition, from which it follows that $\tau_{m,k}(\tilde{\Delta}_{m,k})$ is independent.

\begin{definition}
    We now notate set partitions of $[m]$ using {\em minimal block numbering}; that is, we denote a set partition $A$ of $[m]$ by $(a_1, \ldots, a_m)$, where $a_1 = 1$ and $a_i \le \max(a_1, \ldots, a_{i-1}) + 1$ for $i > 1$. Notice that each set partition has a unique minimal block numering. This induces a {\em lexicographic ordering of set partitions of $[m]$}, where two partitions are lexicographically compared by their minimal block numberings viewed as elements of $\mb{Z}^m$. The {\em initial set partition} of an element $\sum_{A \in \Pi_{m,k}} c_A A$ is the lexicographically minimal set partition appearing in this sum with nonzero coefficient.
\end{definition}

Now take a standard minor symbol $\delta(\alpha \mid \beta) \in \tilde{\Delta}_{m,k}$, and let $a_i$ (resp. $b_i$) be the $i$-th smallest element of $\alpha \setminus \beta$ (resp. $\beta \setminus \alpha$). The initial set partition of $\tau_{m,k}(\delta(\alpha \mid \beta))$ has the following blocks:
\begin{enumerate}
    \item a singleton block for each $c \in (\alpha \cap \beta) \setminus 1$;
    \item a two-element block $\{a_i, b_i\}$ for each $i \le \ab{\alpha \setminus \beta}$; and
    \item a block consisting of 1 and all elements of $[m] \setminus (\alpha \cup \beta)$.
\end{enumerate}
Notice that since $\alpha \le \beta$, the smaller element in a size-2 block (not containing 1) is in $\alpha$ and the larger element is in $\beta$.

\begin{example}
    The initial set partition of $\tau_{8,4}(\delta(1235 \mid 1368))$ is $(1,2,3,1,4,2,1,4)$.
\end{example}

\begin{proposition}
    The $\tau_{m,k}$-images of distinct standard minor symbols in $\tilde{\Delta}_{m,k}$ have distinct initial set partitions.
\end{proposition}
\begin{proof}
    Suppose $\tau_{m,k}(\delta(\alpha \mid \beta))$ and $\tau_{m,k}(\delta(\alpha' \mid \beta'))$ have the same initial set partition $A \in \Pi_{m,k}$ for some $\delta(\alpha \mid \beta), \delta(\alpha' \mid \beta') \in \tilde{\Delta}_{m,k}$; we want to show that these two minor symbols are equal. First, notice that $\alpha \cap \beta = \alpha' \cap \beta'$, since both intersections contain 1 and since the size-1 blocks of $A$ (excluding the block containing 1 if necessary) are exactly the elements of $\alpha \cap \beta \setminus 1$ and $\alpha' \cap \beta' \setminus 1$. Furthermore, notice that $\alpha \cup \beta = \alpha' \cup \beta'$, since removing 1 from the block of $A$ containing 1 yields exactly $[m] \setminus (\alpha \cup \beta)$ and $[m] \setminus (\alpha' \cup \beta')$.
Finally, as noted above for the remaining two element blocks have their smaller element in $\alpha$ and their larger element in $\beta$, repectively in $\alpha'$ and $\beta'$, which determines where to place each element of $[m]$ and hence gives $\alpha=\alpha'$ and $\beta=\beta'$.
\end{proof}

Combined with Propositions \ref{prop:narayana:count} and \ref{prop:narayana:span}, we obtain the following.

\begin{corollary}\label{cor:narayana:basis}
    $\tau_{m,k}(\tilde{\Delta}_{m,k})$ is a basis for $\im \tau_{m,k}$, and $\tau_{m,k}$ has rank $N(m,k)$.
\end{corollary}

In particular, notice that $\tau_{m,k}$ is not surjective unless $k = m-1$, since $N(m,m-1) = \binom{m}{2} = \ab{\Pi_{m,m-1}}$. Surjectivity of $\tau_{m,m-1}$ also implies surjectivity of $\tau_{m,M}$ when $M$ is a multiset all of whose elements are $m-1$, since such a map can be viewed as a symmetric power of $\tau_{m,m-1}$.

\begin{remark}
    Since $\psi_{m,k} = \phi_{m,k} \circ \tau_{m,k}$ and $\phi_{m,k}$ is injective, $\psi_{m,k}$ and $\phi_{m,k}$ convey essentially the same information. For example, the relations described in Propositions \ref{prop:narayana:pivot-rels} and \ref{prop:narayana:flip-rels} are also elements of $\ker \psi_{m,k}$; in other words, the minors of $L_{K_m}$ satisfy not only the Plücker relations, which are quadratic and common to all square matrices, but also these families of linear relations. The results in this section thus constitute {\em combinatorial proofs} of the relations holding between the minors of symmetric nullity-1 matrices (the generic such matrices being $L_{K_m}$).

    The higher degree setting is more complicated. If one compares $\psi_{m,M}$ and $\tau_{m,M}$ when $M$ is a multiset of size at least two, one can no longer say that their kernels or ranges are isomorphic since $\phi_{m,M}$ is no longer injective in general (as evidenced by the existence of spanning forest identities).
\end{remark}

%% file: muir-pf.tex
\section{A proof of Theorem \ref{thm:mtt-det-identities:muir}}\label{header:muir-pf}

Although Muir's law of extensible minors is usually stated for generic matrices, we need the case of Laplacians of complete graphs, so we briefly give the proof of Theorem \ref{thm:mtt-det-identities:muir} here, following the approach of \cite{bs}. Note that the complete Laplacian case does not trivially follow from the generic case because we need that all minor relations of the complete Laplacian extend, not just those that are generic minor relations.

Let $n > m$, and let
 \begin{align*}
        X &= \begin{bmatrix}
            H & G \\
            F & E \\
        \end{bmatrix}
 \end{align*}
 be an $n\times n$ matrix with $E$ an invertible $(n-m) \times (n-m)$ block.  Recall that the \emph{Schur complement} of $E$ is
$H' = H - GE^{-1}F$
and 
\begin{align*}
  \begin{bmatrix}
    I_{m} & -GE^{-1} \\
    0 & I_{n-m} \\
  \end{bmatrix} \begin{bmatrix}
    H & G \\
    F & E
  \end{bmatrix} &= \begin{bmatrix}
    H' & 0 \\
    F & E
  \end{bmatrix},
\end{align*}
so for equicardinal $\alpha, \beta \subseteq [m]$, 
\begin{align*}
  \det H'(\alpha \mid \beta) &= \frac{\det X(\alpha \mid \beta)}{\det E}.
\end{align*}
In particular, for $i, j \in [m]$,
\begin{equation}\label{eq entries of H'}
  h'_{ij} = \frac{\det X([m] \setminus i \mid [m] \setminus j)}{\det E}.
\end{equation}
where $H' = [h'_{ij}]$.

Note that if $X$ has entries in $F[\mathbf{x}]$ for some set of independent indeterminates $\mathbf{x}$, then \eqref{eq entries of H'} immediately implies that $H'$ has entries in $F(\mathbf{x})$ and that if $X$ is symmetric then so is $H'$.

\begin{lemma}\label{lem sum to 0}
  Let $L$ be the Laplacian of a connected graph $G$, let $H$ be its bottom-right $(n-m)\times (n-m)$ block, and let $E$ be its top-left $m\times m$ block. Then $E$ is invertible and the Schur complement $H'$ has all rows and columns summing to $0$.
\end{lemma}

\begin{proof}
  First, notice that $E = L([m] \mid [m])$, so $\det E$ is the sum of all spanning forests of a connected graph with $m$ trees in which each element of $[m]$ is in a distinct tree. In particular, at least one such forest exists, so $\det E$ is non-zero. Thus, $E$ is invertible.

  As observed before the statement of the lemma, $H'$ is symmetric since $L$ is and so it suffices to verify that the rows of $H'$ sum to $0$.  By \eqref{eq entries of H'} and the all-minors matrix-tree theorem, we have that for distinct $i, j \in [m]$,
    \[
        h'_{ij} = \frac{\det L([m] \setminus i \mid [m] \setminus j)}{\det E} 
                = \frac{\eps([m] \setminus i, [n]) \eps([m] \setminus j, [n]) \eps(\sigma_{ij}) \Phi_m^{(G)}(A_{ij})}{\Phi_m^{(G)}((1,2,\ldots,m))},
    \]
    where $\sigma_{ij}$ is the matching $[m] \setminus i \to [m] \setminus j$ that sends $j$ to $i$ and fixes all other elements, and $A_{ij}$ is the set partition of $[m]$ with $m-1$ blocks in which $i$ and $j$ share a block. Furthermore, notice that for $i \in [m]$, $\det L([m] \setminus i \mid [m] \setminus i)$ is the spanning forest polynomial with $m-1$ trees in which each element of $[m]\setminus i$ is in a distinct tree.  So, extending the set partitions, we have
    \[
        h'_{ii} = \frac{\det L([m] \setminus i \mid [m] \setminus i)}{\det E} 
                = \frac{\sum_{j \in [m] \setminus i} \Phi_m^{(G)}(A_{ij})}{\Phi_m^{(G)}((1,2,\ldots,m))}.
    \]
    Notice that for distinct $i, j \in [m]$, we have that modulo 2,
    \begin{align*}
        \eta([m] \setminus i, [n]) + \eta([m] \setminus j, [n]) + \eta(\sigma_{ij}) &= \binom{m+1}{2} - i + \binom{m+1}{2} - j + (i-j+1) \equiv 1,
    \end{align*}
    which implies
    \begin{align*}
        \eps([m] \setminus i, [n]) \eps([m] \setminus j, [n]) \eps(\sigma_{ij}) = -1.
    \end{align*}
    Thus, for each $i \in [m]$,
    \begin{align*}
        \sum_{j \in [m]} h'_{ij} &= \sum_{j \in [m] \setminus i} h'_{ij} + h'_{ii} = 0.
    \end{align*}
  
\end{proof}

Write $L_n$ for the Laplacian of $K_n$.

\begin{lemma}\label{lem:muir-pf:workhorse}
    Let $E$ and $H$ be as in Lemma~\ref{lem sum to 0} for $G=K_n$, and write $H'=[h'_{ij}]$.
    There exists a ring homomorphism $F[E(K_m)] \to F(E(K_n))$ that maps the $(i,j)$-th entry of $L_m$ to $h'_{ij}$.
\end{lemma}
\begin{proof}
    The Laplacian of a complete graph is the generic symmetric matrix with all row and column sums equal to 0. In particular, up to overall sign, the entries above the main diagonal are independent indeterminates and the others determined by those conditions.  Lemma~\ref{lem sum to 0} and the observations above it tell us that $H'$ is also symmetric with entries in $F(E(K_n))$ and with all rows and columns summing to $0$ and so the map $x_{ij}\mapsto h'_{ij}$ is the desired homomorphism.
\end{proof}

\begin{proof}[Proof of Theorem \ref{thm:mtt-det-identities:muir}]
    Suppose we have a homogeneous element of $\ker \Psi_m$; write it as
    \begin{align*}
        \sum_{q=1}^{s} c_q \prod_{i=1}^{t} \delta\pr{\alpha^{(q,i)} \mid \beta^{(q,i)}}.
    \end{align*}
    So, by definition of $\ker \Psi_m$,
    \begin{align*}
        0 &= \sum_{q=1}^{s} c_q \prod_{i=1}^{t} \det L_m \pr{\alpha^{(q,i)} \mid \beta^{(q,i)}}.
    \end{align*}
    Applying the map from Lemma \ref{lem:muir-pf:workhorse} and letting $H'$ be the Schur complement of the bottom-right $(n-m)$-by-$(n-m)$ block of $L_n$, we see that
    \begin{align*}
        0 = \sum_{q=1}^{s} c_q \prod_{i=1}^{t} \det H' \pr{\alpha^{(q,i)} \mid \beta^{(q,i)}} 
          = \sum_{q=1}^{s} c_q \prod_{i=1}^{t} \frac{\det L_n \pr{\alpha^{(q,i)} \mid \beta^{(q,i)}}}{\det L_n([m] \mid [m])} \tag{by \eqref{eq entries of H'}}
    \end{align*}
    Multiplying both sides by $\det L_n([m] \mid [m]))$, we see that our element is in the kernel of $\Psi_n \circ \mu_m^{(n)}$ as well.
\end{proof}

The same proof works for the generic symmetric case but is strictly simpler, as the analogue of Lemma~\ref{lem:muir-pf:workhorse} follows directly from the initial elementary observations on the symmetry of $H'$, without needing Lemma~\ref{lem sum to 0}.